\documentclass[11pt,A4paper]{article}

\usepackage[left=32mm,right=32mm,top=30mm,bottom=32mm]{geometry}
\usepackage{labelfig}
\usepackage{epsfig}
\usepackage{epstopdf}

\usepackage{caption}

\usepackage{xfrac}

\usepackage[percent]{overpic}

\usepackage{color}
\usepackage{amsthm,amsmath,amssymb}
\usepackage{booktabs}
\usepackage{microtype}
\usepackage{overpic}
\usepackage{bm}
\usepackage{sectsty}
\usepackage[
	pdftitle={PDFTitle},
	pdfauthor={Hugo Parlier},
	ocgcolorlinks,
	linkcolor=linkred,
	citecolor=linkred,
	urlcolor=linkblue]
{hyperref}

\usepackage{mathtools}
\usepackage{marginnote}

\definecolor{linkred}{RGB}{199,21,133}
\definecolor{linkblue}{RGB}{16, 78, 139}

\usepackage[hang,flushmargin]{footmisc}
\usepackage{enumerate}
\usepackage{titlesec}
	\titlespacing{\section}{0pt}{12pt}{0pt}
	\titlespacing{\subsection}{0pt}{6pt}{0pt}
	
\titlelabel{\thetitle.\quad}

\makeatletter 

\long\def\@footnotetext#1{%
\H@@footnotetext{%
\ifHy@nesting 
\hyper@@anchor{\@currentHref}{#1}%
\else 
\Hy@raisedlink{\hyper@@anchor{\@currentHref}{\relax}}#1%
\fi 
}}

\def\@footnotemark{%
\leavevmode 
\ifhmode\edef\@x@sf{\the\spacefactor}\nobreak\fi 
\H@refstepcounter{Hfootnote}%
\hyper@makecurrent{Hfootnote}%
\hyper@linkstart{link}{\@currentHref}%
\@makefnmark 
\hyper@linkend 
\ifhmode\spacefactor\@x@sf\fi 
\relax 
}%

\ifFN@multiplefootnote%
\renewcommand*\@footnotemark{%
\leavevmode 
\ifhmode 
\edef\@x@sf{\the\spacefactor}%
\FN@mf@check 
\nobreak 
\fi 
\H@refstepcounter{Hfootnote}%
\hyper@makecurrent{Hfootnote}%
\hyper@linkstart{link}{\@currentHref}%
\@makefnmark 
\hyper@linkend 
\ifFN@pp@towrite 
\FN@pp@writetemp 
\FN@pp@towritefalse 
\fi 
\FN@mf@prepare 
\ifhmode\spacefactor\@x@sf\fi 
\relax%
}%
\fi 

\makeatother 

\newtheorem{theorem}{Theorem}[section]

\newtheorem{lemma}[theorem]{Lemma}
\newtheorem{proposition}[theorem]{Proposition}

\newtheorem*{mainthm1}{Theorem~\ref{thm:qch}}
\newtheorem*{mainthm2}{Theorem~\ref{thm:main2}}

\theoremstyle{definition}

\newtheorem{defn}[theorem]{Definition}
\newtheorem{eg}[theorem]{Example}

\theoremstyle{remark}

\renewcommand{\phi}{\varphi}

\newcommand{\ssm}{\smallsetminus}

\newcommand{\diam}{{\rm diam}}

\newcommand{\area}{{\rm area}}

\newcommand{\injrad}{{\rm injrad}}

\newcommand{\arcsinh}{{\rm arcsinh}}
\newcommand{\arccosh}{{\rm arccosh}}

\newcommand{\be}{ \begin{equation} }
\newcommand{\ee}{ \end{equation} }

\newcommand{\co}{\colon\thinspace}

\sectionfont{\large \bfseries}
\subsectionfont{\normalsize}

\long\def\symbolfootnote[#1]#2{\begingroup%
\def\thefootnote{\fnsymbol{footnote}}\footnote[#1]{#2}\endgroup}

\def\blfootnote{\xdef\@thefnmark{}\@footnotetext}

\date{\today}

\begin{document}

{\Large \bfseries Bounded geometry with no bounded pants decomposition}

{\large 
Ara Basmajian\symbolfootnote[1]{Supported by a grant from the Simons foundation (359956, A.B.)}, Hugo Parlier\symbolfootnote[2]{Supported by the Luxembourg National Research Fund OPEN grant O19/13865598}, and Nicholas G.~Vlamis\symbolfootnote[3]{Supported in part by PSC-CUNY Award \# 63524-00 51.
\vspace{.1cm} 
} 
\vspace{0.5cm}

{\bf Abstract.} 
We construct a quasiconformally homogeneous hyperbolic Riemann surface---other than the hyperbolic plane---that does not admit a bounded pants decomposition.
Also, given a connected orientable topological surface of infinite type with compact boundary components, we construct a complete hyperbolic metric on the surface that has bounded geometry but does not admit a bounded pants decomposition. 

\vspace{.5cm}

\section{Introduction}
A Riemann surface \( X \) is \emph{quasiconformally homogenous}, or \emph{QCH}, if there exists a real number \( K \) such that \( K \geq 1 \) and such that for any two points \( x,y \in X \) there exists a \( K \)-quasiconformal homeomorphism \( f \co X \to X \) with \( f(x) = y \). 
This notion was introduced in \cite{BCMT} for hyperbolic manifolds of arbitrary dimension.
In the same article, the authors gave a characterization of all QCH hyperbolic manifolds in dimension greater than two, which relied on rigidity results that do not exist for Riemann surfaces.
It is an open question to characterize QCH Riemann surfaces.

Building on results in the literature, the first and third author in \cite{BasmajianVlamis} showed that this problem can be split into four topological cases.
Moreover, they gave a characterization in one of the cases: every two-ended infinite-genus QCH Riemann surface with no planar ends is quasiconformally equivalent to a (geometric) regular cover of a closed surface. 
A key component of the proof was to show that every two-ended infinite-genus QCH Riemann surface admits a \emph{bounded pants decomposition}, that is, a pants decomposition in which (non-cuspidal) cuffs have lengths that are universally bounded from above and below.
The motivation of this article is the following question: \textit{Does every hyperbolic QCH Riemann surface, other than the hyperbolic plane, admit such a pants decomposition?}
Our main theorem answers this question in the negative:

\begin{mainthm1}\label{thm:main1}
There exists a non-simply connected quasiconformally homogenous hyperbolic Riemann surface without a bounded pants decomposition.
\end{mainthm1}

Intuitively, every non-simply connected QCH hyperbolic Riemann surface must have its injectivity radius uniformly bounded from above and below (this is shown in \cite[Theorem~1.1]{BCMT}). 
In particular, every non-simply connected QCH hyperbolic Riemann surface has bounded geometry.
A hyperbolic Riemann surface \( X \) has \emph{bounded geometry} if the injectivity radius at each point of \( X \) outside the area 2 horoball neighborhoods of the cusps is universally bounded from above and below.

It is clear that a Riemann surface with a bounded pants decomposition has bounded geometry; however, the converse is false---Kinjo gives a counterexample in \cite{Kinjo1}. 
(However, Kinjo shows that every hyperbolic Riemann surface with bounded geometry does admit a bounded hexagonal decomposition \cite{Kinjo2}.)
In fact, the Riemann surface \( R \) obtained by removing the lattice \( \mathbb Z \oplus i\mathbb Z \) from \( \mathbb C \) is such a counterexample (which is distinct from Kinjo's example, but similar in nature). 
The point is that the thick part of \( R \) is quasi-isometric to \( \mathbb R^2 \) with the standard Euclidean metric and there are curves in any pants decomposition of \( R \) that must bound (topological) disks in \( \mathbb C \) containing an increasing number of lattice points.
The isoperimetric inequality now guarantees that the lengths of these curves tend to infinity. 

Both this example and Kinjo's are based on planar surfaces where every simple closed curve is separating. It is not so clear how to create examples of non-planar Riemann surfaces with bounded geometry that do not admit a bounded pants decomposition.
Our second theorem shows this is always possible:

\begin{mainthm2}\label{thm:main2}
Every infinite-type orientable connected topological surface with compact boundary components admits a complete hyperbolic metric with bounded geometry and such that every pants decomposition has unbounded cuff lengths.
\end{mainthm2}

The fact that, for any topological surface of infinite type, one can find a hyperbolic metric that does not admit a bounded pants decomposition is not entirely surprising in light of the study of the Bers constant for finite-type surfaces. This constant, first studied by Bers \cite{Bers} and quantified by Buser and others \cite{BuserBook,Balacheff-Parlier, Balacheff-Parlier-Sabourau, Parlier}, is an upper bound on a shortest pants decomposition of a hyperbolic surface in terms of its (finite) topology. The lower bounds for this constant are known to grow in terms of topology, but the examples with genus generally use very thin surfaces, so such a limiting surface would not have bounded geometry. 

Also note that if a surface has a bounded pants decomposition then its length spectrum is indiscrete, and hence any hyperbolic surface with a discrete length spectrum cannot have a bounded pants decomposition.
In \cite{Basmajian-Kim}, Basmajian--Kim give a general strategy for constructing hyperbolic surfaces with discrete length spectrum and prescribed topology. 
However, the surfaces we construct in Theorem~\ref{thm:main2} necessarily do not have a discrete length spectrum. 

To establish the above theorems, it is thus necessary to find a new strategy. The key is to mimic the example \( R \) described above, that is, to find surfaces with an isoperimetric profile that will guarantee the non-existence of bounded pants decompositions. Most of our constructions begin with a ``plane with holes": a surface similar to $R$ but with boundary geodesics of a fixed length, which are then pasted to obtain different topological types and in ways which ensure the non-bounded pants decomposition property.

\subsection*{Outline}
After a preliminary section, Section \ref{sec:prelim}, where we set notation and terminology, we study in Section~\ref{sec:pants} the geometry of pairs of pants, and in particular their intrinsic diameters, which will be a key tool in our proofs. In Section~\ref{sec:planes} we establish the basis for our constructions and establish Theorem~\ref{thm:qch}. In the final section, Section \ref{sec:final}, we handle all remaining topological types, establishing Theorem~\ref{thm:main2}. 
\subsection*{Acknowledgements}
We are grateful to Edward Taylor and Richard Canary for asking us whether Theorem~\ref{thm:qch} was true, and more generally if bounded geometry implies existence of a bounded pants decomposition. 
The authors acknowledge support from U.S. National Science Foundation grants DMS 1107452, 1107263, 1107367 ``RNMS: Geometric structures And Representation varieties” (the GEAR Network).

\section{Preliminaries}\label{sec:prelim}
All our surfaces will be orientable. A \emph{hyperbolic surface with totally geodesic boundary} is a connected complete metric space in which every point has a neighborhood isometric to an open subset of a closed geodesic half-plane in the hyperoblic plane. 
Given a hyperbolic surface \( X \), the length of a curve \( \alpha \) in \( X \) is denoted by \( \ell_X(\alpha) \). 
For $x\in X$, $\injrad_X(x)$ will denote the injectivity radius of $X$ in $x$, that is, the supremum of the radii of isometrically embedded hyperbolic disks centered at \( x \). 
The \emph{systole} of \( X \), denoted \( \mathrm{sys}(X) \), is the infimum of the lengths of closed geodesics in \( X \). 
The surface \( X \) is said to have \emph{bounded geometry} if there exists a constant \( M > 1 \) such that \( \mathrm{sys}(X) \geq 1/M \) and \( \injrad_X(x) \leq M \) for every \( x \in X \). 
Equivalently, in the case where \( X \) has no boundary, there are positive universal upper and lower bounds on the injectivity radius of all points of \( X \) in the complement of well chosen neighborhoods of the cusps of \( X \). 

A geodesic pants decomposition of a (hyperbolic) surface is a collection of disjoint simple closed curves such that the complementary regions are all three-holed spheres (pairs of pants) with finite area. For finite-type surfaces this is equivalent to being a maximal collection of disjoint simple closed geodesics that is maximal with respect to inclusion, but for infinite-type surfaces, an extra condition on local finiteness is necessary. A (geodesic) pants decomposition \( \mathcal P \) of \( X \) is \emph{bounded} if there exists a constant \( B > 1 \) such that every closed geodesic \( \gamma \in \mathcal P \) satisfies \( 1/B \leq \ell_X(\gamma) \leq B \). 
Note that not every hyperbolic surface has a geodesic pants decomposition, e.g., if it contains a funnel or half-plane.



In general, the arguments that follow are mostly self contained.
For basic facts and formulas in hyperbolic geometry, we make appropriate references to \cite{BuserBook}.
Outside of basic hyperbolic geometry and surface topology, we will make reference to topological ends and the classification of surfaces based on the end spaces of surfaces, which we refer the reader to Richards \cite{Richards}. 

\section{Bounds on the diameter of pairs of pants}\label{sec:pants}
One of the keys to providing examples of hyperbolic surfaces with no bounded pants decompositions is to show that there are pairs of pants that span long portions of the surface, that is, there are pants with large diameter. 
In order to do this, we need to relate the lengths of the cuffs of a pair of pants to its diameter, and in particular give upper bounds on the diameter as a function of its cuff lengths.
We will also have need to work with pairs of pants with cusps, in which case we will focus on the diameters of the thick part. 
The goal of this section is to establish these upper bounds.
We first treat the case of no cusps (Lemma~\ref{lem:diameter}) and then the cusped case (Lemma~\ref{lem:diametercusps}).
The arguments are similar in nature, but the latter is more technical.

\begin{lemma}
\label{lem:diameter}
Let $l$ and $L$ be positive real numbers satisfying $l<L$. 
There exists a constant \( K = K(l) \) such that any pair of pants $P$ with cuff lengths between $l$ and $L$ has diameter less than \(\frac{3}{2} L + K \).\end{lemma}

\begin{proof}
We first make the observation that every right-angled hexagon has area \(\pi \) and hence there is a bound on the radius of the largest inscribed circle. It follows that the distance between any point in the hexagon and the boundary is bounded by a universal constant. Explicitly:
 for any point $p$ in the hexagon, a ball of area $r$ around $p$ is of area $2\pi(\cosh(r)-1)$, so if the ball is embedded then
$ 2\pi(\cosh(r)-1) < \pi$ and hence $r < \arccosh(3/2)<1$. 
 
We next decompose \( P \) into the union of two isometric right-angled hexagons with disjoint interiors.
We work with one of the hexagons \( H \), which contains three pairwise non-adjacent sides of lengths \( \frac {\ell_1}2, \frac {\ell_2}2, \) and \( \frac {\ell_3}2 \). Denote the length of the side opposite the side of length \( \frac {\ell_3}2 \) by \( w \). 
Using a standard formula for hyperbolic right-angled hexagons (see \cite[Theorem 2.4.1(i)]{BuserBook}), we readily deduce that \[ \cosh w \leq \frac{1}{\sinh^{2} \frac {\ell}2}
\left(\cosh \frac{L}{2}+\cosh^2 \frac{\ell}{2} \right).\]

Now using the facts (for all $x,y\geq 1$) that 
\(\arccosh \,x < \log x+ \log 2 \) and \(\log(x+y) < \log x + \log(y+1) \)
we have 
\[w< \log \cosh \frac{L}{2} + K_1(\ell)< L+ K_1(\ell),\]
where
\[K_1(\ell)=\log \left[\frac{2}{\sinh^{2} \frac {\ell}2} \right] + \log\left(\cosh^2 \frac {\ell}{2} +1 \right) +
\log 2.\]
Of course the same bound holds for the side opposite \(\frac{\ell_1}{2}\) as well as the side opposite \(\frac{\ell_2}{2}\).
 Putting the bounds for each of the sides together we see that the boundary of \( H \) has length bounded by 
 \(3 L + 3 K_1. \) 
 
 Now using the fact that any point on the pair of pants is distance at most 1 from the boundary of \( H \), we have that the diameter of 
 \( P \) is at most \( \frac{3}{2} L + K, \)
where \( K=\frac{3}{2}K_1 + 2\).
\end{proof}

We now adapt this result to the setting of pants with cusps. The goal this time will be to bound the diameter of a thick part of the pants.
To do so we remove a canonical neighborhood of each cusp; in particular, given a pair of pants \( P \) with cusps, the \emph{truncated pair of pants} associated to \( P \) is the subset \( P^t \) of \( P \) obtained by removing the area 2 horoball neighborhoods of each cusp. 
Note that these neighborhoods are always embedded and disjoint.
In fact, these are the maximal such neighborhoods, and this is the sense in which they are canonical (see \cite[Section~4.1]{BuserBook}).
Up to isometry, there is a unique pair of pants with three cusps, and its diameter can be computed explicitly.
However, the 3-cusped pair of pants will not play a role in the sequel, and so we will only consider pairs of pants with at least one totally geodesic boundary component. 

In what follows, we will frequently refer to the \emph{collar function} \( \eta \co (0,\infty) \to (0, \infty) \) defined by
$$
\eta(\ell) = \arcsinh\left(\frac{1}{\sinh(\ell/2)}\right).
$$

\begin{lemma}
\label{lem:diametercusps}
Let $l$ and $L$ be positive real numbers satisfying $l<L$ and let \( P \) be a pair of pants with either one or two cusps and totally geodesic boundary.
Then, there exists a constant \( K = K(l) \) such that if the lengths of the cuffs of \( P \) are in the interval \( [l,L] \), then \( \diam(P^t) \leq L + K \).
\end{lemma}

\begin{proof}
We first treat the case when $P$ has one cusp and two totally geodesics boundary components of length $\ell_1$ and $\ell_2$ in the interval $[l,L]$. 

As in Lemma~\ref{lem:diameter}, we begin by decomposing \( P \) into two right-angled pentagons with disjoint interiors, which when intersected with \( P^t \) yield two isometric right-angled hexagons with disjoint interiors.
Let \( \widetilde H \) be one of the pentagons, and \( H = \widetilde H \cap P^t \) the corresponding hexagon. 
The side of \( \partial H \) corresponding to the horocycle is not geodesic, but a curve of constant curvature of length $1$. 
Together with sides of lengths $\frac{\ell_1}{2}$ and $\frac{\ell_2}{2}$, they form alternating sides of the hexagon. 
We denote the other lengths by $x,y$ and $z$ as in Figure~\ref{fig:Hex1}.

\begin{figure}[h]
\leavevmode \SetLabels
\L(.377*0.1) $\ell_1/2$\\%
\L(.58*0.1) $\ell_2/2$\\%
\L(.58*.54) $x$\\%
\L(.42*.51) $y$\\%
\L(.5*0.03) $z$\\%
\L(.51*.74) $1$\\%
\endSetLabels
\begin{center}
\AffixLabels{\centerline{\includegraphics[width=4cm]{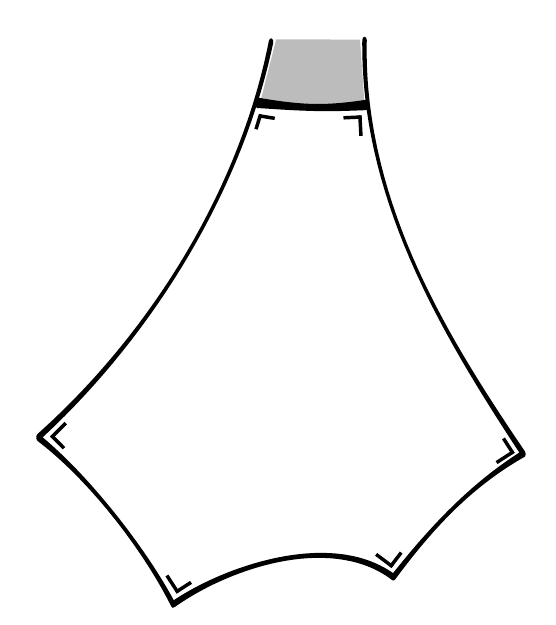}}}
\vspace{-24pt}
\end{center}
\caption{The truncated hexagon $H$ inside the pentagon \( \widetilde H \).}
\label{fig:Hex1}
\end{figure}
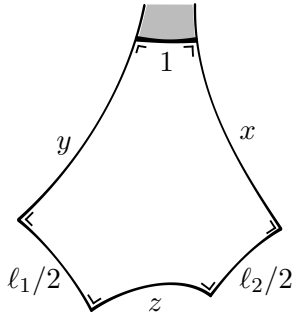

As in the previous lemma, we observe that any point $p \in H$ is distance at most 1 from $\partial H$. (The argument only depended on an upper bound of the area of the hexagon and here $\area(H) = \pi-1$.) 

As before, a bound on the perimeter of \( H \), denoted \( \ell(\partial H) \), yields a bound on the diameter of \( P^t \).
Indeed, any two points of \( P^t \) are distance at most 1 from \( \partial H \), and hence $\diam(P^t) \leq \frac{\ell(\partial H)}{2}+2$. In order to bound this perimeter, we proceed to estimate $x,y$ and $z$. 

The pentagon \( H \) can be further decomposed into two Lambert quadrilaterals with disjoint interiors by taking the perpendicular from the side of \( H \) of length \( z \) to the cusp.
Using a standard formula for Lambert quadrilaterals (see \cite[Theorem 2.3.1(i)]{BuserBook}), we have 
$$
z = \eta(\ell_1) + \eta(\ell_2) \leq 2\, \eta(\ell).
$$
The goal is now to show that there is a bound on both $x$ and $y$ of the form $\frac{L}{2} + \Delta$, where $\Delta$ only depends on $\ell$. 
There are certainly many ways of doing this, including writing down an explicit formula for $x$ and $y$ that depend on $\ell_1$ and $\ell_2$, and then studying the resulting function. 
To simplify things, we instead view $P$ as the geometric limit of a pair of pants where one of the cuff lengths goes to $0$, which will allow us to use a standard trigonometric formula.

Let \( \ell_0 \) be a positive real number, and let $P_{\ell_0}$ be a pair of pants with cuff lengths $\ell_0, \ell_1$ and $\ell_2$. 
As before, decompose \( P_{\ell_0} \) into two isometric right-angled hexagons with disjoint interiors; let \( H_{\ell_0} \) be one of these hexagons.
By the collar lemma (see \cite[Theorem 4.1.1]{BuserBook}), the cuff of length $\ell_0$ admits an embedded neighborhood of width \( \eta(\ell_0) \) in \( P_{\ell_0} \) that is disjoint from the other cuffs. 
Let \( P_{\ell_0}^t \) be the result of removing the \( \eta(\ell_0) \)-neighborhood around the cuff of length $\ell_0$ from \( P_{\ell_0} \), and let \( H^t_{\ell_0} = P^t_{\ell_0} \cap H_{\ell_0} \). 

Keeping with our notation, let \( x_0 \) and \( y_0 \) denote the lengths of the sides in \( H^t_{\ell_0} \) opposite the sides of length \( \ell_1/2 \) and \( \ell_2/2 \), respectively.
The area of \( \eta(\ell_0) \)-neighborhood in \( P_{\ell_0} \) is \( \ell_0\sinh(1/\ell_0) \), which limits to 2 as \( \ell_0 \) tends to 0. 
In particular, as \( \ell_0 \to 0 \), the pair of pants \( P_{\ell_0} \) limits to the pair of pants \( P \), and the \( \eta(\ell_0) \)-neighborhood around the cuff of length \( \ell_0 \) limits to the area 2 horoball neighborhood of the cusp in \( P \). 
It follows that \( x_0 \to x \) and \( y_0 \to y \) as \( \ell_0 \to 0 \).

The side of \( H_{\ell_0} \) containing the side of \( H^t_{\ell_0} \) of length \( x_0 \) has length \( x_0 + \eta(\ell_0) \).
It follows that
\begin{eqnarray*}
\cosh\left(x_0 + \eta(\ell_0) \right) &=& \frac{\cosh(\ell_1/2) + \cosh(\ell_0/2) \cosh(\ell_2/2)}{\sinh(\ell_0/2) \sinh(\ell_2/2)}\\
&\leq & \frac{\cosh(L/2) + \cosh(\ell_0/2) \cosh(\ell/2)}{\sinh(\ell_0/2) \sinh(\ell/2)}
\end{eqnarray*}
where the equality is an immediate application of a standard formula for hexagons (see \cite[Theorem~2.4.1(i)]{BuserBook}), and where the inequality readily follows from the fact that the hyperbolic cotangent function is monotonically decreasing on the positive real line.
Hence, 
$$
x_0 \leq \arccosh\left( \frac{\cosh(L/2) + \cosh(\ell_0/2) \cosh(l/2)}{\sinh(\ell_0/2) \sinh(l/2)} \right) -\eta(\ell_0).
$$
Now $\arccosh(t) = \log(t+\sqrt{t^2-1}) < \log(2 t)$ and $\arcsinh(t) = \log(t+\sqrt{t^2+1}) > \log(2 t)$ and thus
\begin{align*}
x_0 	&< \log\left( \frac{1}{\sinh(\ell_0/2)} \left( \frac{\cosh(L/2)}{\sinh(l/2)}+ \cosh(\ell_0/2) \coth(l/2)\right) \right) - \log\left( \frac{1}{\sinh(\ell_0/2)}\right) \\
	&= \log\left( \frac{\cosh(L/2)}{\sinh(l/2)}+ \cosh(\ell_0/2) \coth(l/2)\right) \\
	&\leq \log\left( \frac{\cosh(L/2)+ \cosh(\ell_0/2) \cosh(L/2)}{\sinh(l/2)}\right) \\
	&= \log\left( {\cosh(L/2)}\right) + \log\left(\frac{1+ \cosh(\ell_0/2)}{\sinh(\ell/2)}\right) \\
	&\leq \frac L2 +\log\left(\frac{1+ \cosh(\ell_0/2)} {\sinh(\ell/2)}\right)
\end{align*}
where the third inequality uses the fact that \( l \leq L \) and the fifth inequality uses the fact that \( \log(\cosh t) < t \) whenever \( t> 0 \). 
By continuity, we have
\begin{align*}
x	&= \lim_{\ell_0\to 0} x_0 \\
	&\leq \lim_{\ell_0 \to 0} \left(\frac L2 +\log\left(\frac{1+ \cosh(\ell_0/2)}{\sinh(\ell/2)}\right) \right) \\
	&= \frac L2 + \log\left(\frac2{\sinh(\ell/2)}\right)
\end{align*}

The same argument applies to $y$. As $\ell_1$ and $\ell_2$ are bounded above by $L$, we obtain an upper bound on the perimeter of $H$:
$$
\ell(\partial H) < x+y+z+\ell_1/2+\ell_2/2 +1 < 2 L + 2\, \eta(\ell)+ 2 \log\left( \frac{2}{\sinh(l/2)}\right) +1.
$$
We now take the function $K(l)$ to be 
$$
K(l) = \eta(\ell) + \, \log\left( \frac{3}{\sinh(l/2)}\right) + 2 +\frac{1}{2}
$$
and thus we obtain 
$$
\diam(P^t) \leq L + K(l)
$$
as desired. 

We now treat the case when $P$ has two cusps; the method is similar but less complicated than the one-cusp case. 
As before, we consider a truncated pair of pants, but this time there are two horocyclic neighborhoods removed. 
We denote by $\ell_1$ the length of the only geodesic boundary curve, which we assume lies between $l$ and $L$. 
We look at the corresponding truncated hexagon $H$, which this time has an axial symmetry. 
The remaining side lengths are denoted as in Figure \ref{fig:Hex2}. 

\begin{figure}[h]
\leavevmode \SetLabels
\L(.48*0.01) $\ell_1/2$\\%
\L(.57*.33) $x$\\%
\L(.41*.33) $x$\\%
\L(.46*0.58) $1$\\%
\L(.52*.58) $1$\\%
\L(.493*.86) $z$\\%
\endSetLabels
\begin{center}
\AffixLabels{\centerline{\includegraphics[width=4cm]{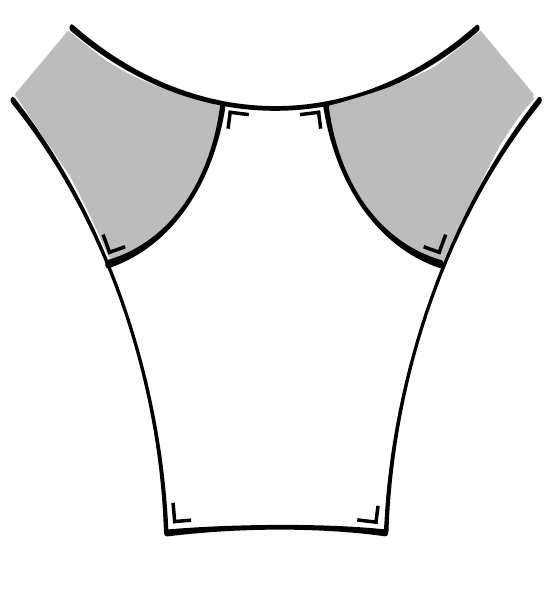}}}
\vspace{-24pt}
\end{center}
\caption{The doubly truncated hexagon $H$.}
\label{fig:Hex2}
\end{figure}
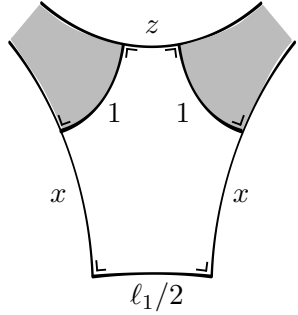

As before, we have $$
\diam(P^t) \leq \ell(\partial H)/2 + 2,
$$
and so need to bound the perimeter of \( H \).
We use the same limiting argument as before by considering a pair of pants with two boundary geodesics of length $\ell_0$ and the third of length $\ell_1$. 
We then remove the two standard collars of width $\eta(\ell_0)$ around the two geodesics of length $\ell_0$ and denote by $z_0$ the common orthogonal between the two collars (corresponding to $z$ on $H$). 
Working in the hexagon \( H_{\ell_0} \), another straightforward application of \cite[Theorem~2.4.1(i)]{BuserBook} yields
\[
\cosh\left(z_0 + 2\eta(\ell_0) \right) = \frac{\cosh(\ell_1/2) + \cosh^2(\ell_0/2)}{\sinh^2(\ell_0/2)} \leq \frac{\cosh(L/2) + \cosh^2(\ell_0/2)}{\sinh^2(\ell_0/2)},
\]
and hence
\[
z_0 \leq \arccosh\left( \frac{\cosh(L/2) + \cosh^2(\ell_0/2)}{\sinh^2(\ell_0/2)} \right) - 2 \eta(\ell_0).
\]
Using the same functional properties as before, we can deduce that
\begin{align*}
z_0	&\leq \log\left( \frac{\cosh(L/2) + \cosh^2(\ell_0/2)}{\sinh^2(\ell_0/2)} \right) - 2 \log\left( \frac{1}{\sinh(\ell_0/2)}\right)\\
	&= \log\left( \cosh(L/2) + \cosh^2(\ell_0/2)\right).
\end{align*}
Using continuity, the fact that \( \log(\cosh t) < t \) for \( t> 0 \), and a basic calculus computation, we obtain
\begin{align*}
z	&= \lim_{\ell_0\to0} z_0 \\
	&\leq \lim_{\ell_0\to0} \log\left( \cosh(L/2) + \cosh^2(\ell_0/2)\right) \\
	&= \log(\cosh(L/2) + 1) \\
	&\leq \frac L2 + 1
\end{align*}

We now need to bound the quantity $x$. 
\begin{figure}[h]
\leavevmode \SetLabels
\L(.48*0.01) $\ell_1/2$\\%
\L(.493*.86) $z$\\%
\L(.413*0.74) $1$\\%
\L(.41*.27) $x$\\%
\endSetLabels
\begin{center}
\AffixLabels{\centerline{\includegraphics[width=4cm]{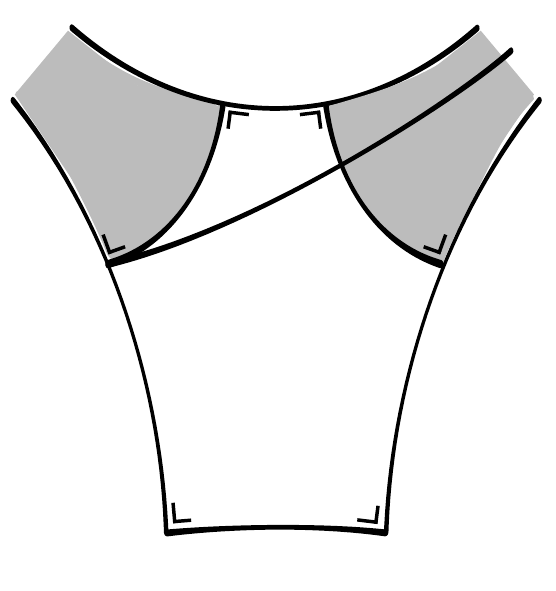}}}
\vspace{-24pt}
\end{center}
\caption{A Lambert quadrilateral with an ideal point and a right-angled ideal triangle.}
\label{fig:Hex3}
\end{figure}
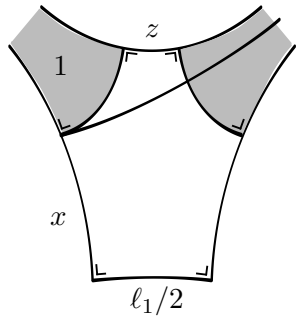
As in Figure \ref{fig:Hex3}, we consider the unique orthogeodesic coming from the opposite ideal point and note the base point of this orthogeodesic is exactly the endpoint of the boundary of the horocyle. This can either be computed, or in fact easily observed in Figure \ref{fig:x1}.

\begin{figure}
\centering
\includegraphics{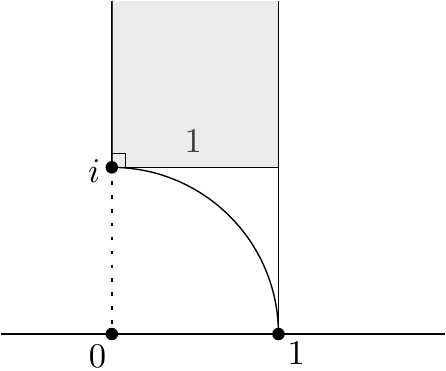}
\caption{The orthogeodesic and length $1$ boundary of the horocyle are tangent.}\label{fig:x1}
\end{figure}

Again using a standard formula for Lambert quadrilaterals (\cite[Theorem~2.3.1(i)]{BuserBook}), we have
$$
x= \eta(\ell_1)\leq \eta(\ell).
$$

We can now bound the perimeter of \( H \):
\[
\ell(\partial H) = 2x+2+z+ \frac L2 \leq L + 2\, \eta(\ell) + 3
\]
and it follows that the diameter of $P^t$ satisfies 
$$
\diam(P^t)<\frac{L}{2} + K(l)
$$
for 
$$
K(l) = \eta(\ell) + \frac32.
$$
\end{proof}

The above lemmas will be used to show that any pants decomposition of certain surfaces have pants with arbitrarily long cuff lengths. 
We make the following observation that will be crucial in how we use the above lemmas. 
Let $P$ be an embedded pair of pants in a hyperbolic surface $X$. 
Then, $P$ has an intrinsic distance, which we denote $d_P$, but also a distance coming from $X$, denoted $d_X$. 
When applicable, the same observation applies to a truncated pair of pants $P^t$. 
Hence, we have the following inequalities:

\begin{equation}\label{eq:pants}
\sup_{p,q \in P} d_X(p,q) \leq \sup_{p,q \in P} d_P(p,q) \leq \diam(P)
\end{equation}

and

\begin{equation}\label{eq:trunc_pants}
\sup_{p,q \in P^t} d_X(p,q) \leq \sup_{p,q \in P^t} d_P(p,q) \leq \diam(P^t).
\end{equation}

\section{Planes with handles and QCH Riemann surfaces}\label{sec:planes}

The main goal of this section is to give an example of a quasiconformally homogeneous surface that fails to admit a bounded pants decomposition and hence establishing Theorem~\ref{thm:qch}. 
This surface will be an example of a larger class of surfaces---\emph{planes with handles}---constructed below in Section~\ref{sec:plane_handles}.
Topologically, planes with handles are constructed from the plane by removing open disks and identifying the resulting boundary components in pairs. 
Geometrically, we will explore two cases: either the distance between any two identified pairs is universally bounded or not.
Section~\ref{sec:bounded_gluings} deals with the former case and Section~\ref{sec:unbounded_gluings} the latter. 
Before considering these cases, in Section~\ref{sec:quadrangle} we establish a topological lemma about arcs on surfaces and a general geometric proposition giving a lower bound on the cuff lengths in a pants decomposition based on the existence of \emph{quadrangular subsurfaces}.

\subsection{Planes with handles}
\label{sec:plane_handles}

Given a real number \( b \) such that \( \sinh(b) > 1 \), there exists a unique right-angled hyperbolic pentagon with adjacent sides of length \( b \) \cite[Lemma 2.3.5]{BuserBook}.
Pasting four copies together, we obtain a one-holed square \( R_b \) as shown in Figure~\ref{fig:one-square}.
If we set \( b = \arcsinh(1) \), then the pentagon degenerates into an ideal Lambert quadrilateral, and in this case, we can again paste four copies together to obtain a one-cusped square \( R_{\arcsinh(1)} \). 
In either case, given \( b \geq \arcsinh(1) \), we can paste copies of \( R_b \) together to form a grid with an action of \( \mathbb Z^2 \) by isometries; we denote the resulting surface by \( \Sigma_b \).
Observe that if \( b > \arcsinh(1) \), then \( \Sigma_b \) is one-ended and has infinitely many boundary components; if \( b = \arcsinh(1) \), then the end space of \( \Sigma_b \) is homeomorphic to a convergent sequence with its limit point; in either case, \( \Sigma_b \) is often referred to as a \emph{flute surface}. 

\begin{defn}
A hyperbolic surface obtained from \( \Sigma_b \) by identifying a subset of boundary components pairwise via orientation-reversing isometries is called a \emph{plane with handles}. 
If the distance between any pair of identified boundary components is universally bounded, then we say that the associated plane with handles has \emph{bounded gluings}. 
\end{defn}

\begin{figure}
\centering
\includegraphics{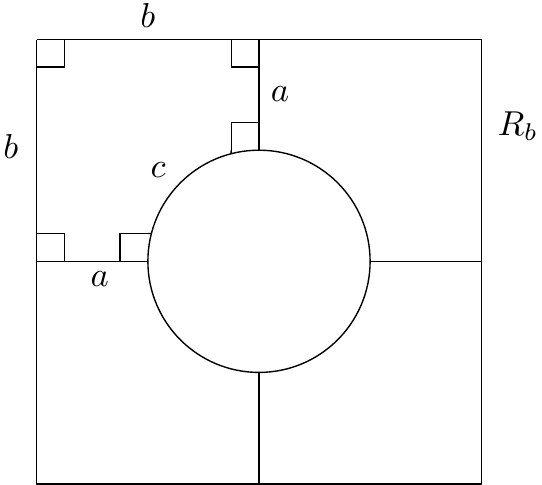}
\caption{The one-holed hyperbolic square \( R_b \) made of four right-angled pentagons.}
\label{fig:one-square}
\end{figure}

For example, \( \Sigma_b \) is itself a plane with handles with bounded gluings (since there are no gluings); this example also shows that the naming is a bit of a misnomer since \( \Sigma_b \) has no handles. 
We readily see that a plane with handles has bounded geometry. In Proposition~\ref{prop:bounded-gluing} below, we will see that a plane with handles with bounded gluings does not admit a bounded pants decomposition. 

\begin{figure}
\centering
\includegraphics[width=4.5cm]{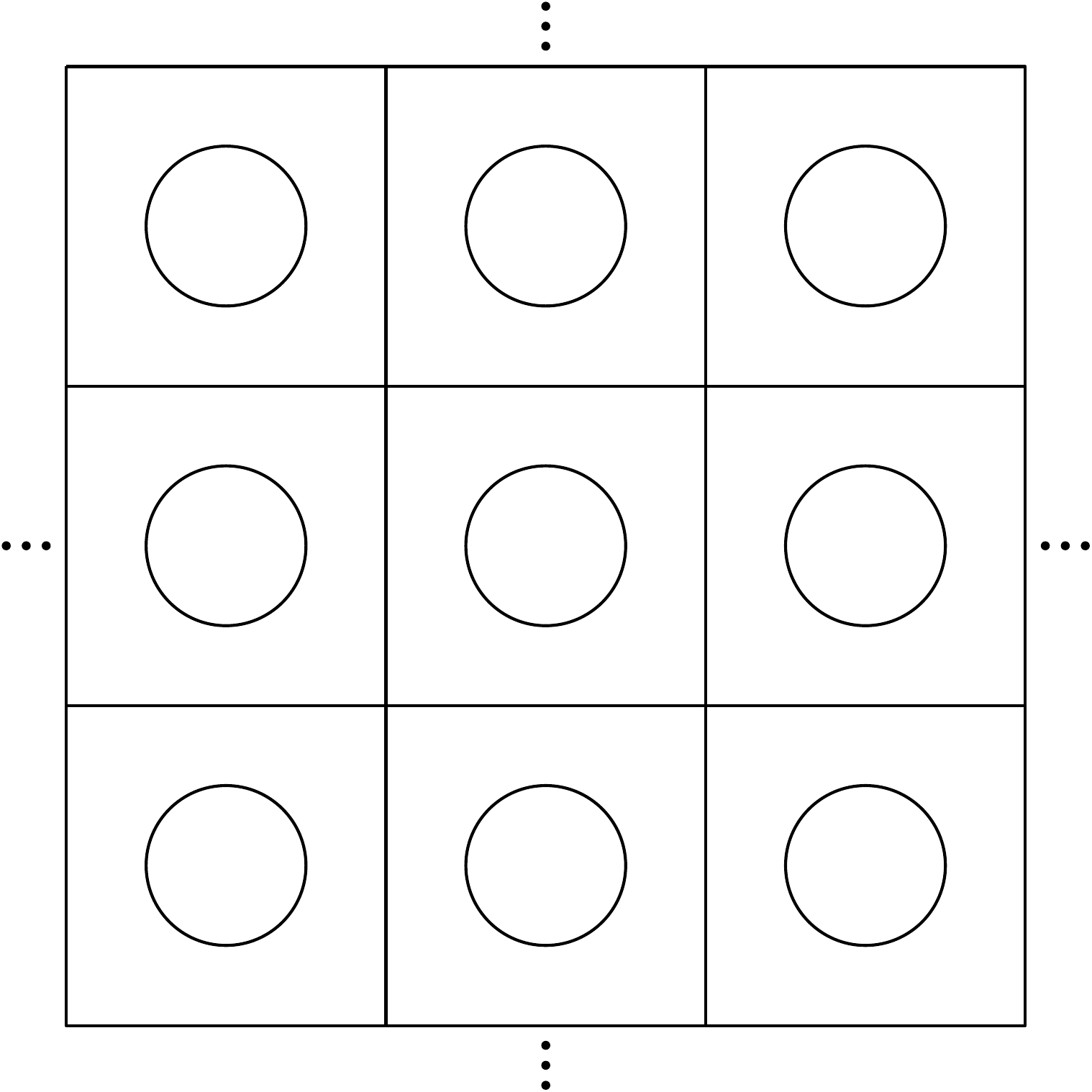}
\caption{The grid made by pasting one-holed squares.}
\label{fig:grid}
\end{figure}

In what follows, it will be helpful for us to have a fixed embedding of \( \Sigma_b \) into \( \mathbb R^2 \), so that we can refer to coordinates. 
Fix an isometric copy of \( R_b \) in \( \Sigma_b \) and choose a side of \( R_b \).
The complete geodesics obtained as concatenations of the sides of the squares are called {\it horizontal} and {\it vertical} according to their position in Figure \ref{fig:grid}.

Now fix an embedding of \( \Sigma_b \) into \( \mathbb R^2 \) such that the restriction of the embedding to every proper horizontal and vertical geodesic of \( \Sigma_b \) is an isometry and such that the origin in \( \mathbb R^2 \) sits at the intersection of a horizontal and vertical geodesic of \( \Sigma_b \). 
The action of \( \mathbb Z^2 \) on \( \mathbb R^2 \) given by \( (m,n) \cdot (x,y) = (x+2mb, y+2nb) \) restricts to an action by isometries on \( \Sigma_b \). 
Pick a boundary component of \( \Sigma_b \), label it \( \partial_{0,0} \), and set \( \partial_{i,j} = (i,j)\cdot \partial_{0,0} \); this is a labelling of all the boundary components of \( \Sigma_b \).

\subsection{Surfaces with quadrangular boundary}
\label{sec:quadrangle}

The goal in this subsection is to give a general criterion on a hyperbolic surface---in terms of the existence of specific subsurfaces---that gives a lower bound on the supremum of lengths in a pants decompositions of the surface.
The reader should keep planes with handles in mind below; however, we will work in greater generality so we may use the results later.

\begin{defn}
Let \( X \) be a hyperbolic surface with a preferred boundary component \( \delta \) and four marked points \( w,x,y,z\in \delta \).
We say \( X \) has \emph{quadrangular boundary} if each component of \( \delta \ssm \{w,x,y,z\} \) is a geodesic segment and each component of \( \partial X \ssm \delta \) is totally geodesic; we call \( \delta \) the \emph{quadrangular boundary component}.
Let \( a,b,c, \) and \( d \) be the closures of the components of \( \delta \ssm \{w,x,y,z\} \) labelled such that \( a \cap c = \varnothing \).
Then, we define the \emph{width} of \( X \) to be the quantity
\[
\omega(X) = \min\{ \ell_X(\alpha) : \alpha\colon\thinspace [0,1] \to X \text{ such that } \alpha(0) \in a, \alpha(1) \in c \text{ or } \alpha(0)\in b, \alpha(1) \in d \}.
\]
\end{defn}

Our first goal is to show that if a hyperbolic surface \( X \) contains a subsurface with quadrangular boundary of large width and long quadrangular boundary, then any pants decomposition of \( X \) must have large cuffs.
In order to do this, we need a topological lemma. 

\begin{lemma}
\label{lem:arcs}
Let \( S \) be a surface with a preferred boundary component with four marked points.
Let \( a,b,c, \) and \( d \) denote the closure of the complementary components of the four marked points in the preferred boundary component of \( S \), labelled such that \( a \cap c = \varnothing \).
Let \( \mathcal A \) be a collection of essential and pairwise-disjoint simple closed curves and simple arcs on \( S \) such that each arc has two distinct endpoints on the preferred boundary component of \( S \).
If \( \mathcal A \) contains finitely many arcs, none of which connect \( b \) and \( d \), then there exists a path in \( S \) connecting \( a \) and \( c \) that is disjoint from every arc and curve in \( \mathcal A \). 
\end{lemma}

\begin{proof}
If \( \mathcal A \) has no arcs, then the boundary of \( S \) is a connected component of the complement of the union of curves in \( \mathcal A \), and hence there is a path from \( a \) to \( c \) as desired, namely \( b \).
We may therefore assume that there is at least one arc in \( \mathcal A \). 

First consider the case in which no arc in \( \mathcal A \) has an endpoint in \( a \cup c \), and hence every arc either has both endpoints in \( b \) or both endpoints in \( d \). 
Let us focus on \( b \); let \( \mathcal A_b \) be the arcs in \( \mathcal A \) with both endpoints in \( b \). 
We say two arcs \( \delta, \eta \in \mathcal A_b \) \emph{overlap} if the endpoints of \( \delta \) separate the endpoints of \( \eta \) in \( b \) (and vice versa). 
Observe that if no two arcs of \( \mathcal A_b \) overlap (and in particular if \( |\mathcal A_b| \in \{0, 1\} \)), then there is a component of a regular neighborhood of \( b \cup (\bigcup_{\delta \in \mathcal A_b} \delta) \) connecting \( a \) to \( c \). 

We will argue now by induction on the number of arcs in \( \mathcal A_b \). 
We have already completed the base cases, that is, when \( |\mathcal A_b| \in \{0,1\} \). 
Let \( k \in \mathbb N \) and suppose \( |\mathcal A_b| = k +1 \).
If no two of the arcs in \( \mathcal A_b \) overlap, then as we have already argued, there exists a path from \( a \) to \( c \).
Otherwise, we can choose \( \delta, \eta \in \mathcal A_b \) that overlap. 
Let \( S_\delta \) be the surface obtained by cutting \( S \) along \( \delta \); observe that \( S_\delta \) has two boundary components.
Since \( \delta \cap \eta = \varnothing \), we may view \( \eta \) as an arc in \( S_\delta \); moreover, since the endpoints of \( \delta \) separate the endpoints of \( \eta \) in \( b \), and hence in \( \partial S \), we must have that the endpoints of \( \eta \) are on distinct components of \( \partial S_\delta \). 
It follows that the surface \( S_\delta^\eta \) obtained by cutting \( S_\delta \) along \( \eta \) has a single boundary component.

The four marked points on \( \partial S \) determine four points on \( \partial S_\delta^\eta \) whose complement has four components, three of which correspond to \( a, c, \) and \( d \) from \( \partial S \), and the other is a union of arcs from \( b \) together with \( \delta \) and \( \eta \).
We let \( \mathcal A' \) denote the set of arcs in \( S_\delta^\eta \) obtained by viewing each arc in \( \mathcal A \ssm \{\delta, \eta\} \) as an arc in \( S_\delta^\eta \), and we similarly define \( \mathcal A_b' \). 
Then, \( |\mathcal A_b'| = k-1 \), and hence, by our induction hypothesis, there exists a path connecting \( a \) to \( c \) in \( S_\delta^\eta \) disjoint from the union of the arcs in \( \mathcal A' \).
Note that this path is also a path in \( S \), which concludes the proof in the case where no arc in \( \mathcal A \) has an endpoint in \( a \cup c \).

Now, assume there are arcs in \( \mathcal A \) with endpoints in \( a \cup c \).
Let \( \mathcal A' \) be the subset of \( \mathcal A \) consisting of arcs with no endpoints in \( a\cup c \), and let \( \alpha \) be the path from \( a \) to \( c \) disjoint from each arc and curve in \( \mathcal A' \) constructed in the previous case.
Let \( \mathcal A_\alpha \) denote the collection of arcs in \( \mathcal A \) that have nontrivial intersection with \( \alpha \).
Then, there is a component of the regular neighborhood of the union of \( \alpha \) with every arc in \( \mathcal A_\alpha \) that is a path from \( a \) to \( c \) disjoint from every arc and curve of \( \mathcal A \). 
\end{proof}


Given a complete hyperbolic surface \( X \), a hyperbolic surface \( Y \) with quadrangular boundary, and an embedding of \( Y \) into \( X \), we define the \emph{width of \( Y \) relative to \( X \)} to be the quantity
\[
\omega_X(Y) = \inf\{ \ell_X(\alpha) : \alpha\colon\thinspace [0,1] \to X \text{ such that } \alpha(0) \in a, \alpha(1) \in c \text{ or } \alpha(0)\in b, \alpha(1) \in d \},
\]
where \( a, b, c, \) and \( d \) are the closures of the components of the complement of the four marked points in the quandrangular boundary of \( Y \), labelled such that \( a \cap c = \varnothing \). 

In order to adapt the standard definition to our setting, we define the \emph{systole} of a hyperbolic surface to be the infimum of lengths of closed geodesics. 

\begin{proposition}
\label{prop:big-cuffs}
Let \( X \) be a complete hyperbolic surface with positive systole \( \ell \), and let \( Y \) be a hyperbolic surface with quadrangular boundary. 
If there exist an embedding of \( Y \) into \( X \) such that the quadrangular boundary component \( \delta \) of \( Y \) is separating and essential, then any pants decomposition of \( X \) contains a curve of length at least \( \min\{ \ell_X(\widehat \delta), \frac{2}{3}\left( \omega_X(Y)-K\right)\} \), where \( \widehat \delta \) is the unique closed geodesic in \( X \) homotopic to \( \delta \) and \( K = K(\ell) \) is as in Lemma \ref{lem:diameter}.
\end{proposition}

\begin{proof}
Fix a pants decomposition of \( X \).
If \( \widehat \delta \) is a curve in the decomposition, then we are finished; so, assume this is not the case.
Let \( Z \) denote the closure of the component of \( X \ssm \delta \) containing \( Y \). 
The result of intersecting each curve in the given pants decomposition with \( Z \) yields a collection \( \mathcal A \) of pairwise-disjoint simple geodesic arcs and simple closed geodesics.
Only finitely many curves in a pants decomposition intersect a given compact set, and hence there are only finitely many arcs in \( \mathcal A \).
Let \( a,b,c \), and \( d \) be the closure of the complementary components of the four marked points in \( \delta \), labelled such that \( a \cap c = \varnothing \).
We can then apply Lemma~\ref{lem:arcs} to see that there is a path in \( Z \) disjoint from the arcs and curves in \( \mathcal A \) connecting either \( b \) and \( d \) or \( a \) and \( c \).
Either way, this path must have length at least \( \omega_X(Y) \). 
Note that this path must be contained in a pair of pants in the decomposition; in particular, this pair of pants must have diameter at least \( \omega_X(Y) \), or if the pair of pants has a cusp, then the associated truncated pair of pants has diameter at least \( \omega_X(Y) \).
In either case, by Lemma~\ref{lem:diameter} or Lemma~\ref{lem:diametercusps}, we can conclude that this pair of pants must have a cuff of length at least \( \frac{2}{3}\left( \omega_X(Y)-K\right) \) as desired. 
\end{proof}

\subsection{Planes with handles with bounded gluings}
\label{sec:bounded_gluings}

Given \( m \in \mathbb N \), let \( S^b_m \) denote an \( m^2 \)-holed square obtained by pasting \( m^2 \) copies of \( R_b \) in an \( m \times m \)-grid.
Let the \emph{outer boundary} of \( S^b_m \) refer to the unique non-smooth boundary component of \( S^b_m \); in particular, \( S^b_m \) is a hyperbolic surface with quadrangular boundary, whose quadrangular boundary component is the outer boundary and whose marked points correspond to the four non-smooth points of the outer boundary. 
We readily see that \( \omega(S^b_m) = 2mb \). 

%
%

\begin{proposition}
\label{prop:bounded-gluing}
If \( X \) is a plane with handles with bounded gluings, then every pants decomposition of \( X \) has unbounded cuff lengths. 
\end{proposition}

\begin{proof}
Let us first consider the case with no gluings, that is, when \( X = \Sigma_b \) for some \( b \geq \arcsinh(1) \). 
First note that the systole \( \ell \) of \( X \) is positive. 
Moreover, for every \( m \in \mathbb N \), \( X \) contains an isometric and convex copy of \( S^b_m \); in particular, \( \omega_{\Sigma_b}(S^b_m) = \omega(S^b_m) = 2mb \). 
Also note that if \( m > 1 \), then the geodesic in \( X \) homotopic to the outer boundary of \( S^b_m \) has length at least \( 8(m-1)b > 2mb \). 
Hence, by Proposition~\ref{prop:big-cuffs}, given any pants decomposition of \( X \), it contains a curve of length at least \( 2mb-K(\ell) \), and thus the lengths of the cuffs diverge.

Now, let us assume that \( X \) is obtained by gluing boundary components of \( \Sigma = \Sigma_b \) for some \( b > \arcsinh(1) \). 
We will proceed to use the coordinate system on \( \Sigma \) described previously.
%

Let \( \sim \) denote the equivalence relation on \( \mathbb Z^2 \) given by \( (i,j) \sim (i',j') \) if \( \partial_{i,j} \) and \( \partial_{i',j'} \) are identified in \( X \). 
Since \( X \) has bounded gluings, we can define
\[
 J= \max\{ n \in \mathbb N : \text{ there exists } (i,j) \in \mathbb Z^2 \text{ s.t. } (i, j) \sim (i, j+n) \text{ or } (i,j)\sim(i+n,j) \}.
\]

Let \( m \in \mathbb N \) satisfy \( m > 2J \). 
Let \( S_m \subset \Sigma \) denote the intersection of \( \Sigma \) with the square in \( \mathbb R^2 \) with vertices \( (0,0), (2mb, 0), (0,2mb), \) and \( (2mb,2mb) \), so that \( S_m \) is isometric to \( S_m^b \). 
Now, we do not know if the outer boundary of \( S_m \) is separating in \( X \) or not; however, the key to the proof is that the outer boundary is close to a separating curve, as we now describe. 
Let \( S_m^J \subset \Sigma \) be the intersection of \( \Sigma \) with the square in \( \mathbb R^2 \) containing \( S_m \) with side lengths \( 2mb+2J \) and having vertex \( (-J,-J) \). 
It follows that if \( 0 \leq i,j \leq m \) and \( (i,j) \sim (i',j') \), then \( \partial_{i',j'} \subset S_m^J \). 
We can therefore find a separating closed geodesic \( \eta_m \) in \( \Sigma \) contained in \( S_m^J \) and such that every non-outer boundary component of \( S_m \) is on the same side of \( \eta_m \). 
To see this, for each \( (i',j') \in \mathbb Z^2 \) such that \( (i',j') \sim (i,j) \) with \( \partial_{i,j} \subset S_m \) and \( \partial_{i',j'} \not\subset S_m \), choose a path in \( S_m^J \ssm S_m \) connecting the outer boundary of \( S_m \) and \( \partial_{i',j'} \) with the stipulation that any two such paths are disjoint.
The geodesic homotopic to the boundary component of a regular neighborhood of the union of these paths with the outer boundary of \( S_m \) and the \( \partial_{i',j'} \) that are disjoint from \( S_m \) has the desired property.
Before continuing, note that \( \ell_X(\eta_m) > 8(m-1)b \).

Let \( Y \) denote the subsurface of \( X \) bounded by \( \eta_m \) that contains each non-quadrangular boundary component of \( S_m \).
Let \( w,x,y,z \in \eta_m \) be points in the intersection of \( \eta_m \) with the Euclidean diagonals of the square defining \( S_m^J \) such that at exactly one of these points is contained in each of the four \( (J+1) \times (J+1) \)-squares sharing a vertex with \( S_m^J \) and intersecting \( S_m \). 
With these marked points, we can view \( Y \) as a hyperbolic surface with quadrangular boundary.

Let \( d \) denote the minimum distance in \( \Sigma \) between any two of its boundary components. 
Observe that \( \eta_m \) can only intersect \( S_m \) in the copies of \( R_b \) that share a vertex with \( S_m \). 
It follows that, by construction, \( \omega_X(Y) \geq 2d(m/J-2) \).
Let \( \ell \) denote the systole of \( X \), and note that \( \ell > 0 \). 
We can therefore apply Proposition~\ref{prop:big-cuffs} to see that any pants decomposition of \( X \) contains a curve of length at least 
$$
 \min\bigg\{8(m-1)b, \frac{2}{3}\left(2d(m/J-2)-K\right)\bigg\}.
 $$
Since this holds for all \( m > 2J+2 \), we see that any pants decomposition of \( X \) has unbounded cuff lengths. 
\end{proof}

To finish this subsection, we consider a particular example in order to establish the main theorem, Theorem~\ref{thm:qch}. 
Using the notation setup in Proposition~\ref{prop:bounded-gluing}, let \( X \) be a plane with handles such that the induced equivalence relation \( \sim \) on \( \mathbb Z^2 \) is given as follows: \( (2i, j) \sim (2i+1, j) \) for all \( i,j \in \mathbb Z \). 
We then see that the \( \mathbb Z^2 \) action on \( \Sigma_b \) descends to an action on \( X \) by isometries, and moreover, \( \mathbb Z^2 \backslash X \) is a closed surface of genus two.
Every geometric regular cover of a closed surface is QCH \cite[Proposition~2.7]{BCMT}, and hence \( X \) is a QCH plane with handles with bounded gluings.
Our main theorem now follows from applying Proposition~\ref{prop:bounded-gluing} to \( X \):

\begin{theorem}
\label{thm:qch}
There exists a non-simply connected quasiconformally homogenous hyperbolic Riemann surface without a bounded pants decomposition. 
\end{theorem}

\subsection{Planes with handles with unbounded gluings}
\label{sec:unbounded_gluings}

In Proposition~\ref{prop:bounded-gluing}, we saw that every pants decomposition of every plane with handles with bounded gluings has unbounded cuffs.
Here, we will see that if the gluings are unbounded, then there is no such uniform statement: in particular, there are examples of planes with handles with unbounded gluings that have bounded pants decompositions and examples without.

\begin{figure}[h]
\centering
\includegraphics{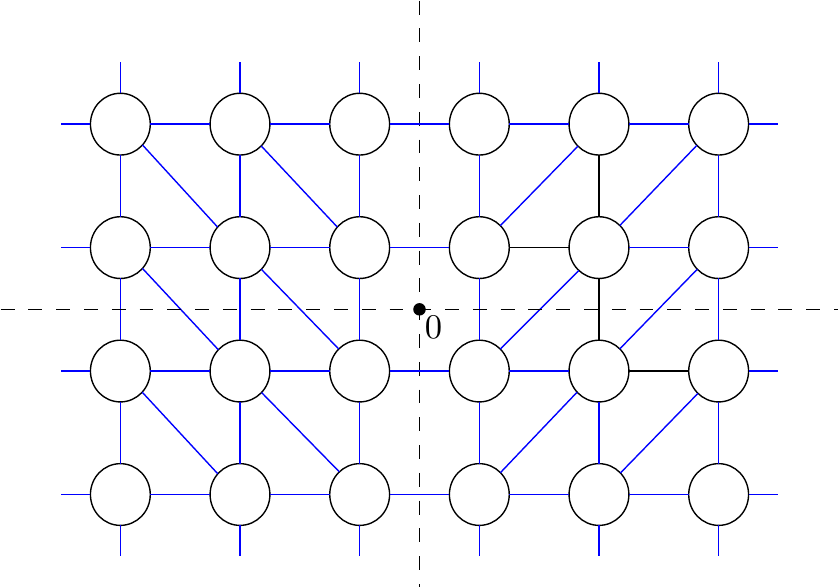}
\caption{The arcs (in blue) close up to give a bounded pants decomposition of \( X_\rho \).}
\label{fig:hexagonal}
\end{figure}

\begin{eg}[Unbounded gluing, bounded pants decomposition]
\label{ex1}
Fix \( b > \arcsinh(1) \). 
Observe that the reflection \( \rho \) of \( \mathbb R^2 \) given by \( (x,y) \mapsto (-x,y) \) restricts to an (orientation-reversing) isometry of \( \Sigma_b \).
Let \( \sim \) be the equivalence relation on \( \Sigma_b \) defined by \( (x,y) \sim (x',y') \) if and only if \( (x',y')=(x,y) \) or \( (x,y) \in \partial \Sigma_b \) and \( (x',y') = (-x,y) \). 
Then, \( X_\rho = \Sigma_b / \sim \) is a plane with handles with unbounded gluings. 
However, \( X_\rho \) has a bounded pants decomposition: such a pants decomposition is depicted in Figure~\ref{fig:hexagonal}.
\end{eg}

\begin{eg}[Unbounded gluing, all pants decompositions unbounded]
\label{ex2}
Fix \( b > \arcsinh(1) \).
Define an equivalence relation \( \approx \) on \( \mathbb Z^2 \) as follows: \( (i,j) \approx (i',j') \) if and only if \( j=j' \), and
\begin{itemize}
\item \( i=i' \), or
\item \( j < 0 \) and \( i'=-i \), or
\item \( j \geq 0\), \( |i-i'|=1 \), and \( \min\{i,i'\} \) is even.
\end{itemize}
Let \( X \) be obtained from \( \Sigma_b \) by identifying \( \partial_{i,j} \) with \( \partial_{i',j'} \) by an orientation-reversing isometry whenever \( (i,j) \approx (i',j') \). 
The portion of \( X \) obtained from the lower half of \( \Sigma_b \) looks like the surface described in Example~\ref{ex1}, and in particular, the second condition guarantees that \( X \) has unbounded gluings.
However, the portion of \( X \) obtained from the upper half of \( \Sigma_b \) looks like the QCH plane with handles described for Theorem~\ref{thm:qch}, and in particular, an identical argument to that of Proposition~\ref{prop:bounded-gluing} applied to this subsurface of \( X \) shows that every pants decomposition has unbounded cuff lengths. 
\end{eg}

\section{Constructions in all topological types}\label{sec:final}

The goal of this section is to show that, given any connected orientable surface \( S \) with compact boundary components and of infinite topological type, there exists a complete hyperbolic metric on \( S \) that has bounded geometry but that does not have a bounded pants decomposition.
In the previous section, we gave a general construction for building such hyperbolic structures on planes with handles; however, topologically, in the empty boundary case, this covers only two distinct surfaces, namely the flute surface and the Loch Ness monster surface.
We now consider the rest.
The construction will split into two cases: first, we will consider finite-genus 2-manifolds whose space of ends is homeomorphic to the union of a Cantor set and a finite discrete set, and then we will consider the rest.
In the second case, we will use the fact that the end space must either have infinitely many isolated planar ends or at least one non-planar end.

\subsection{Finite-type surface minus a Cantor set}

We begin with the case of a sphere minus a Cantor set.
Let \( C \) be a copy of a Cantor set in \( \mathbb S^2 \) and let \( S_C = \mathbb S^2 \ssm C \).
The idea is to write \( S_C \) as a union of flute surfaces, which we individually identify with \( \Sigma_b \) for some \( b > \arcsinh(1) \), and then show that the resulting surface has the desired geometric properties.

To do this, fix a collection \( \{\delta_n\}_{n\in\mathbb N} \) of homotopically nontrivial, pairwise-disjoint, pairwise non-homologous, simple closed curves on \( S_C \) whose images in \( H_1(S_C, \mathbb Z) \) form a basis. 
Let \( \{ S_k \}_{k\in \mathbb N} \) be an enumeration of the closures of the components of the complement of \( \bigcup_{n\in\mathbb N} \delta_n \).
For each \( k \in \mathbb N \), \( S_k \) is topologically a flute surface. 
To see this, note that \( S_k \) is planar, it has infinitely many boundary components, every boundary component is compact, every simple closed curve bounds a compact subsurface, and the surface is one ended.
If the second condition did not hold, then \( S_C \) would either have an isolated end or there would be a linear relation between the curves in homology;
if either of the last two conditions were not satisfied, then the original collection of curves could not have been a homology basis.
The classification of surfaces now guarantees that \( S_k \) is a flute surface. 


We will now endow each of the \( S_k \) with a geometry that will force the surface to not have bounded pants decompositions. 
To do so, we refer to the previous construction via squares (see Figure \ref{fig:one-square}), where the quantities $a,b$ and $c$ are defined as in Section~\ref{sec:plane_handles}. 
We choose $b$ to be such that $c= \frac{1}{2} \arcsinh(1)$, so that the inner cuff length will be $2\,\arcsinh(1)$. Explicitly, using \cite[Theorem~2.3.4]{BuserBook} as before, one chooses $b$ to be such that 
$$
\sinh^2(b)= \cosh\left(\frac{1}{2} \arcsinh(1)\right).
$$
This results in $b$ being exactly
$$
b= \arcsinh\left( \sqrt{\cosh\left(\frac{1}{2} \arcsinh(1)\right)}\right).
$$
In particular, note that \( b > \arcsinh(1) \), and hence, by \cite[Lemma~2.3.5]{BuserBook}, the desired rectangle, \( R_b \), exists. 
We can now equip each \( S_k \) with a hyperbolic metric so that it is isometric to \( \Sigma_b \).

By construction, we can write \( S = \left(\bigsqcup_{n\in \mathbb N} S_k \right) \, / \sim \), where \( \sim \) is an equivalence relation determined by the fact that the \( S_k \) are subsurfaces of \( S_C \).
Now, on a given boundary component, we can realize \( \sim \) as an orientation-reversing isometry, and in doing so, we equip \( S_C \) with a complete hyperbolic metric; let us call the resulting hyperbolic surface \( Z \).

Before working with \( Z \), we need a lemma, which is deduced from \cite[Lemma~4 and its following remark]{Balacheff-Parlier}:

\begin{lemma}
\label{lem:cuff_length}
Let \( F \) be a finite-area planar hyperbolic surface with totally geodesic boundary.
Given a simple closed geodesic $\gamma$ of length at most $2\,\arcsinh(1)$ and a pants decomposition $\mathcal Q$ of \( F \), there exists a pants decomposition containing $\gamma$ whose maximum cuff length is at most the maximum cuff length of $\mathcal Q$.
\end{lemma}

Roughly speaking, the proof consists of constructing a new pants decomposition using arcs obtained from the curves in $\mathcal Q$ by cutting along $\gamma$.
Now by the collar lemma, geodesic arcs that cut across \( \gamma \) must pick up a definitive, and significant, amount of length, so it is natural to expect that adding \( \gamma \) to the pants decomposition will not increase the maximum length of a cuff in the decomposition.
This is accomplished through careful cut and paste arguments together with length estimates.
We note that it is crucial to the argument that the length of \( \gamma \) is at most \( 2\, \arcsinh(1) \) and that the surface is planar (which forces all simple closed curves to be separating).

\begin{proposition}\label{prop:spherecase}
Every pants decomposition of $Z$ has unbounded cuff lengths.
\end{proposition}

\begin{proof}
Each copy of \( \Sigma_b \) associated to the \( S_k \) is isometrically embedded in \( Z \).
Let \( \Sigma \) be one such copy.
We denote by $\Gamma$ the set of boundary curves of $\Sigma$; let \( \Gamma = \{\gamma_m\}_{m\in\mathbb N} \) be an enumeration of \( \Gamma \). 
The main step of the proof will be to show that any pants decomposition of $Z$ can be replaced by a pants decomposition that contains $\Gamma$ without increasing the supremum of lengths of the pants decomposition. 
This new pants decomposition will include a pants decomposition of \( \Sigma \), which by Proposition~\ref{prop:bounded-gluing}, necessarily has unbounded cuff lengths.

We begin with a pants decomposition $\mathcal P$ of $Z$. 
For each $\gamma \in \Gamma$, let $\mathcal P_\gamma$ be the (finite) set of curves of $\mathcal P$ that intersect $\gamma$, and let $Z_\gamma$ be the subsurface of $Z$ spanned by $\mathcal P_\gamma$.

We can now apply Lemma~\ref{lem:cuff_length} with \( Z_{\gamma_1}, \mathcal P_{\gamma_1}, \) and \( \gamma_1 \) to obtain a new pants decomposition, say $\mathcal P_1$, of $Z$ containing \( \gamma_1 \) and such that the supremum of lengths of curves in \( \mathcal P_1 \) is no more than that of \( \mathcal P \). 
Now observe that \( \gamma_1 \) is either peripheral to or disjoint from \( Z_{\gamma_2} \).
In particular, we can again use Lemma~\ref{lem:cuff_length} to obtain a pants decomposition, say $\mathcal P_2$, of $Z$ containing both \( \gamma_1 \) and \( \gamma_2 \) and such that the supremum of lengths of curves in \( \mathcal P_2 \) is no more than that of \( \mathcal P \).
Continuing in this fashion for each \( m \in \mathbb N \), we build a pants decomposition \( \mathcal P_m \) of \( Z \) containing \( \gamma_1, \ldots, \gamma_m \) such that the supremum of cuff lengths is at most that of \( \mathcal P \). 
With the setup, we can take the direct limit of these pants decompositions to obtain a pants decomposition of \( Z \) containing each curve in \( \Gamma \) and whose supremum of cuff lengths is at most that of \( \mathcal P \).
Now, as already noted, it follows from Proposition~\ref{prop:bounded-gluing} that \( \mathcal P \) has unbounded cuff lengths.
\end{proof}

We now treat the topological case of a finite-type surface with a Cantor set removed. 
Let \( F \) be a finite-type surface and let \( F' \) be obtained from \( F \) by removing an open disk.
Fix a hyperbolic metric on \( F' \) so that its boundary is totally geodesic and each component has length \( 2\,\arcsinh(1) \).
Now, let \( \Sigma' \) be one of the isometric copies of \( \Sigma_b \) embedded in \( Z \) (so that \( \Sigma' \) corresponds to \( S_k \) for some \( k \in \mathbb N \)), and let \( Z' \) be a connected component of \( Z \ssm \Sigma' \).
Then, topologically, \( Z' \) is a closed disk with a Cantor set removed from its interior.
Let \( Z_F \) be the complete hyperbolic surface obtained as the union of \( F' \) and \( Z' \) with the boundary of \( Z' \) identified with a component of \( \partial F' \) via an orientation-reversing isometry. 
Topologically, \( Z_F \) is homeomorphic to the surface \( F \) with a Cantor set removed.

\begin{proposition}\label{prop:ftcase}
\label{prop:cantor}
Let \( F \) be an orientable finite-type surface. 
Any pants decomposition of $Z_F$ has unbounded cuff lengths.
\end{proposition}

\begin{proof}
The proof is very similar to the proof of Proposition~\ref{prop:spherecase}; however, we need to carefully choose \( \Sigma \), an embedded copy of \( \Sigma_b \), that is significantly far away from \( F' \).
As a consequence, unlike the proof of Proposition~\ref{prop:spherecase}, we will argue by contradiction.

Suppose there exists a pants decomposition $\mathcal P$ of $Z_F$ with bounded cuff lengths, and let
$$
L= \sup_{\delta\in \mathcal P} \ell(\delta).
$$
By construction and an application of the collar lemma, we can readily see that there exists an isometric embedding of \( \Sigma_b \) into \( Z_F \) whose image \( \Sigma \) (corresponding to one of the \( S_k \)) has  distance  strictly greater than \( L \) from \( F' \).
In particular, given any boundary component \( \gamma \) of \( \Sigma \), the surface spanned by the curves in \( \mathcal P \) intersecting \( \gamma \) is disjoint from \( F' \), and hence must be planar. 
This allows us to proceed as in the proof of Proposition \ref{prop:spherecase} and construct a bounded pants decomposition of \( \Sigma \), which contradicts Proposition~\ref{prop:bounded-gluing}; hence, we can conclude that \( \mathcal P \) has unbounded cuff lengths.
\end{proof}

\subsection{The general setup}

 Let \(S \)  be an infinite-type surface that cannot be obtained by removing a Cantor set from a finite-type surface (this case was considered previously). 
We give a general outline for constructing a complete hyperbolic metric on \( S \) with bounded geometry and such that there are no bounded pants decompositions.

First,  decompose $S$ into flute surfaces by cutting along a collection of appropriate curves \cite[Lemma~3.2]{McCleayParlier}. Each flute surface has a single topological end, which corresponds to a unique end of $S$ (the set of such ends is dense in the end space of \( S \)).

 Next fix \( b>\arcsinh(1) \), and call one of the flute surfaces 
 \(Q_1\). Identify  each  flute surface in the decomposition other than  \(Q_1 \)  with the hyperbolic flute surface \(\Sigma_b \).  On \( Q_1 \) delete a countable sequence of (open) disks with pairwise-disjoint closures such that only finitely many of the disks intersect any given compact set---this surgered surface is again a flute surface. 
Next identify \(\Sigma_b \) with the surgered \( Q_1 \) by assigning the upper-half boundary curves of 
\(\Sigma_b \) with the boundaries of the deleted disks in 
\( Q_1 \). Identify  the boundary curves of the original flute surface 
\( Q_1 \)  with the lower half boundary curves of 
\(\Sigma_b \) in any way.  Denote this hyperbolic flute surface 
by \( Q_1' \) and continue to denote  its boundary curves,
as we did  in Section~\ref{sec:planes}, by \( \partial_{i,j} \).

By the topological assumptions on \( S \), namely that it is not homeomorphic to a finite-type surface with a Cantor set removed, we may assume without loss of generality that the end of \( S \) corresponding to the end of \( Q_1 \) is either non-planar; planar and accumulated by boundary components; or neither, in which case it is planar and accumulated by isolated planar ends. 
The boundary gluing data for the lower half curves of \( Q_1' \)  is determined since they come from the boundary curves of \( Q_1 \).  
Depending on the topology of the end of 
\( Q_1 \), we modify the surgered flute surface \( Q_1' \) in the following way:

\begin{itemize}

\item   If the end of \( Q_1 \) is non-planar, for 
\( j > 0 \) identify \( \partial_{2i,j} \) and \( \partial_{2i+1,j} \) of \( Q_1' \)  via orientation-reversing isometries for all \( i \in \mathbb Z \); 

\item  if the end is non-planar and accumulated by boundary components, then simply leave the boundary components as they are;

\item otherwise, for each \( \partial_{i,j} \) with \( j > 0 \), take a copy of the unique pair of pants with two cusps and a boundary component of length \( b \), and identify its boundary with \( \partial_{i,j} \) via an orientation-reversing isometry. 
 
\end{itemize}

 Reassembling all of the flute surfaces along with the modified flute surface \( Q_1' \),  we obtain a complete hyperbolic surface with the topology of the original surface $S$. 

Finally, there exist ``large squares" (quadrangular subsurfaces of arbitrarily large width) whose distinguished boundary curve is a square in the upper half-plane of the modified \( Q_1'\). These quadrangular subsurfaces come from the large isometric embeddings of \( S_m^b \) in the upper half of \( Q_1' \)---just as in the proof of Proposition~\ref{prop:bounded-gluing}.
The large width of these embeddings can be guaranteed by embedding them arbitrarily far away from the lower half of \( Q_1' \). 
Hence, by Proposition~\ref{prop:big-cuffs}, every pants decomposition of \( X \) must have unbounded cuff lengths.
Now by construction, \( X \) has bounded geometry, and hence we have established:

\begin{proposition}
\label{prop:no-cantor}
Let \( S \) be an orientable connected surface of infinite topological type and with compact boundary components. 
If either \( S \) has infinite genus or the end space of \( S \) is not homeomorphic to a Cantor set union a finite discrete set, then there exists a complete hyperbolic structure on \( S \) with bounded geometry and such that every pants decomposition has unbounded cuff lengths. 
\end{proposition}

By the classification of orientable connected surfaces with compact boundary components, Proposition~\ref{prop:cantor} and Proposition~\ref{prop:no-cantor} establish the second main theorem:

\begin{theorem}
\label{thm:main2}
Every infinite-type orientable connected topological surface with compact boundary components admits a complete hyperbolic metric with bounded geometry and such that every pants decomposition has unbounded cuff lengths.
\end{theorem}


{\it Addresses:}\\ The Graduate Center \& Hunter College, CUNY, N.Y., N.Y., USA\\
Department of Mathematics, University of Luxembourg, Esch-sur-Alzette, Luxembourg\\
Queens College, CUNY, Flushing, N.Y., USA

{\it Emails:}\\
abasmajian@gc.cuny.edu\\
hugo.parlier@uni.lu\\
nicholas.vlamis@qc.cuny.edu

\end{document}